\documentclass[11pt,english]{amsart}
\usepackage[T1]{fontenc}
\usepackage[latin1]{inputenc}
\usepackage{amstext}
\usepackage{amsthm}
\usepackage{amssymb}
\usepackage{graphicx}
\usepackage[usenames]{xcolor}
\usepackage[hmargin=3cm, vmargin=3cm]{geometry}

\makeatletter

\providecommand{\tabularnewline}{\\}

\numberwithin{equation}{section}
\numberwithin{figure}{section}
\theoremstyle{plain}
\newtheorem{thm}{\protect\theoremname}
  \theoremstyle{plain}
  \newtheorem{lem}[thm]{\protect\lemmaname}
  \theoremstyle{plain}
  \newtheorem{prop}[thm]{\protect\propositionname}
  \theoremstyle{remark}
  \newtheorem{rem}[thm]{\protect\remarkname}
  \theoremstyle{definition}
  
  \theoremstyle{plain}
  \newtheorem{cor}[thm]{\protect\corollaryname}
  \newtheorem{thmx}{Theorem}

\makeatother

\makeatother

\usepackage{babel}
  \providecommand{\corollaryname}{Corollary}
  \providecommand{\definitionname}{Definition}
  \providecommand{\lemmaname}{Lemma}
  \providecommand{\propositionname}{Proposition}
  \providecommand{\remarkname}{Remark}
\providecommand{\theoremname}{Theorem}

\begin{document}

\title[The Bolza curve and some orbifold ball quotient surfaces]{The Bolza curve and some orbifold  ball quotient surfaces}

\addtolength{\textwidth}{10mm}
\addtolength{\hoffset}{0mm} 
\addtolength{\textheight}{5mm}
\addtolength{\voffset}{-0mm} 

\global\long\global\long\def\Alb{{\rm Alb}}
 \global\long\global\long\def\Jac{{\rm Jac}}
\global\long\global\long\def\Disc{{\rm Disc}}

\global\long\global\long\def\Tr{{\rm Tr}}
 \global\long\global\long\def\NS{{\rm NS}}
\global\long\global\long\def\PicVar{{\rm PicVar}}
\global\long\global\long\def\Pic{{\rm Pic}}
\global\long\global\long\def\Br{{\rm Br}}
 \global\long\global\long\def\Pr{{\rm Pr}}

\global\long\global\long\def\Hom{{\rm Hom}}
 \global\long\global\long\def\End{{\rm End}}
 \global\long\global\long\def\aut{{\rm Aut}}
 \global\long\global\long\def\NS{{\rm NS}}
 \global\long\global\long\def\SSm{{\rm S}}
 \global\long\global\long\def\psl{{\rm PSL}}
 \global\long\global\long\def\CC{\mathbb{C}}
 \global\long\global\long\def\BB{\mathbb{B}}
 \global\long\global\long\def\PP{\mathbb{P}}
 \global\long\global\long\def\QQ{\mathbb{Q}}
 \global\long\global\long\def\RR{\mathbb{R}}
 \global\long\global\long\def\FF{\mathbb{F}}
 \global\long\global\long\def\DD{\mathbb{D}}
 \global\long\global\long\def\NN{\mathbb{N}}
 \global\long\global\long\def\ZZ{\mathbb{Z}}
 \global\long\global\long\def\HH{\mathbb{H}}
 \global\long\global\long\def\Gal{{\rm Gal}}
 \global\long\global\long\def\OO{\mathcal{O}}
 \global\long\global\long\def\pP{\mathfrak{p}}
 \global\long\global\long\def\pPP{\mathfrak{P}}
 \global\long\global\long\def\qQ{\mathfrak{q}}
 \global\long\global\long\def\mm{\mathcal{M}}
 \global\long\global\long\def\aaa{\mathfrak{a}}
\global\long\def\a{\alpha}
 \global\long\def\b{\beta}
 \global\long\def\d{\delta}
 \global\long\def\D{\Delta}
\global\long\def\L{\Lambda}
 \global\long\def\g{\gamma}
 \global\long\def\G{\Gamma}
 \global\long\def\d{\delta}
 \global\long\def\D{\Delta}
 \global\long\def\e{\varepsilon}
 \global\long\def\k{\kappa}
 \global\long\def\l{\lambda}
 \global\long\def\m{\mu}
 \global\long\def\o{\omega}
 \global\long\def\p{\pi}
 \global\long\def\P{\Pi}
 \global\long\def\s{\sigma}
 \global\long\def\S{\Sigma}
 \global\long\def\t{\theta}
 \global\long\def\T{\Theta}
 \global\long\def\f{\varphi}
 \global\long\def\deg{{\rm deg}}
 \global\long\def\det{{\rm det}}
 \global\long\def\dps{{\displaystyle }}
 \global\long\def\ker{{\rm Ker\,}}
 \global\long\def\im{{\rm Im\,}}
 \global\long\def\rk{{\rm rk\,}}
 \global\long\def\car{{\rm car}}
\global\long\def\fix{{\rm Fix( }}
 \global\long\def\card{{\rm Card\  }}
 \global\long\def\codim{{\rm codim\,}}
 \global\long\def\coker{{\rm Coker\,}}
 \global\long\def\mod{{\rm mod }}
 \global\long\def\pgcd{{\rm pgcd}}
 \global\long\def\ppcm{{\rm ppcm}}
 \global\long\def\la{\langle}
 \global\long\def\ra{\rangle}

\author{Vincent Koziarz, Carlos Rito, Xavier Roulleau}
\begin{abstract}
We study Deraux's non-arithmetic orbifold ball quotient surfaces obtained as
 birational transformations of a quotient $X$ of a particular Abelian surface $A$.
 Using the fact that $A$ is the Jacobian of the Bolza genus $2$ curve, 
we identify $X$ as the weighted projective plane $\PP(1,3,8)$.
We compute the equation of the mirror $M$ of the orbifold ball quotient $(X,M)$
and by taking the quotient by an involution,  we obtain an orbifold ball quotient surface with mirror 
birational to an interesting configuration of plane curves of degrees $1,2$ and $3$. 
We also exhibit an arrangement of four conics in the plane which provides the 
above-mentioned ball quotient orbifold surfaces. \newline $2010$ MSC: 22E40 (14L30 20H15 14J26)
\end{abstract}

\maketitle

\section{Introduction.}
Chern numbers of smooth complex surfaces of general type $X$ satisfy the Bogomolov-Miyaoka-Yau inequality $c_1 ^2(X) \leq 3c_2(X)$.
Surfaces for which the equality 
 is reached
are ball quotient surfaces: there exists a cocompact torsion-free lattice $\Gamma$ in the automorphism
 group $PU(2,1)$ of the ball $B_2$ such that $X=B_2 /\Gamma$. 
This description of ball quotient surfaces by uniformisation is of transcendental nature, and  
in fact among ball-quotient surfaces, very few are constructed geometrically (e.g. by taking cyclic covers 
 of known surfaces or  by explicit equations of an embedding in a projective space).
 
 Among lattices in $PU(2,1)$, only $22$ commensurability classes are known to be non-arithmetic. The first examples of such lattices were given by Mostow and Deligne-Mostow (see \cite{Mo}  and \cite{DM}), and
 recently Deraux, Parker and Paupert \cite{DPP,DPP2} constructed some more, sometimes related to an earlier work of Couwenberg, Heckman and Looijenga~\cite{CHL}. 
 
 Being rare and difficult to produce, these examples are particularly interesting and 
 one would like a geometric description of them.
 To do so, Deraux \cite{Deraux} studies the quotient  of the Abelian surface 
$A=E\times E$, where $E$ is the elliptic curve $E=\CC/\ZZ[i\sqrt{2}]$, by an order $48$ automorphism
group isomorphic to $GL_{2}(\FF_{3})$ that we will denote by $G_{48}$. 
The ramification locus of the quotient map $A\to A/G_{48}$ is the union of $12$ elliptic curves 
and two orbits of 
 isolated  fixed points. 
The images of these two orbits are singularities of type $A_2$ and $\frac{1}{8}(1,3)$, respectively.

Then Deraux proves that (on some birational transforms) 
the 1-dimensional branch locus $M_{48}$ 
of the quotient map $A\to A/G_{48}$ 
and the two singularities are the support of four 
ball-quotient orbifold structures, three of these corresponding to non-arithmetic lattices in $PU(2,1)$. 
Knowing the branch locus $M_{48}$ is therefore important for these ball-quotient orbifolds, since
it gives an explicit geometric description of the uniformisation maps from the ball to the surface.

Deraux also remarks in \cite{Deraux} that the invariants of $A/G_{48}$ and its singularities are the same 
as for the weighted projective plane $\PP(1,3,8)$ and, in analogy with cases in 
\cite{DM2} and  \cite{Deraux2}  where weighted projective planes appear in the context of ball-quotient surfaces, 
he asks whether the two surfaces are isomorphic.

In fact, the quotient $A/G_{48}$ can also be seen as a quotient 
$\CC^2/G$ where $G$ is an affine crystallographic complex reflection group. 
The Chevalley Theorem assert that if $G'$ is a finite reflection group acting on a space $V$ 
then the quotient $V/G'$ is a weighted projective space.
Using theta functions, Bernstein and Schwarzman \cite{BS}   
observed  that for many examples 
of affine crystallographic complex reflection groups $G$ acting on a space $V$,  
 the quotient $V/G$ is
 a also weighted projective space. Kaneko, Tokunaga and Yoshida \cite{KTY} worked 
 out some other cases, and it is believed 
 that this analog of the Chevalley Theorem always happens (see \cite{BS},  \cite[p. 17]{Dolg}), 
 although no general method is known (see also the presentation of the problem 
 given by Deraux in \cite{Deraux},  where more details can be found).  

In this paper we prove that indeed:

\begin{thmx}\label{thm:iso}
The surface $A/G_{48}$ is isomorphic to $\PP(1,3,8).$
\end{thmx}
We obtain this result by exploiting the fact that $A$ is the Jacobian of a smooth genus $2$ 
curve $\t$, a curve which was first studied by Bolza \cite{Bol}. The automorphism group of the curve $\t$ 
induces the action of $G_{48}$ on the Jacobian $A$. The main idea to obtain 
Theorem \ref{thm:iso}  is to understand 
the image of the curve $\t$ in $A$  by the quotient map $A\to A/G_{48}$ 
and to prove that its strict transform in the minimal resolution is a $(-1)$-curve.

We then construct birational 
transformations of $\PP(1,3,8)$ to $\PP^1 \times \PP^1,$ $\PP^2$ 
and obtain the equations of the images 
$M_{\PP^1 \times \PP^1},\, M_{\PP^2}$ of the branch curve $M_{48}$ in these surfaces 
(and also $M_{48}\subset \PP(1,3,8)$). In particular: 

\begin{thmx} \label{thm:uni}
In the projective plane, the mirror $M_{\PP^2}$ is the quartic curve
$$(x^{2}+xy+y^{2}-xz-yz)^{2}-8xy(x+y-z)^{2}=0.$$
This curve   has  two smooth flex points and singular set 
$\mathfrak{a}_1+2\mathfrak{a}_2$ (where an $\mathfrak{a}_{k}$ singularity
has local equation $y^{2}-x^{k+1}=0$).
The line $L_0$ through the two residual points of the flex lines  $F_1,\,F_2$ contains the node (by flex line we mean the tangent line to a flex point).
\end{thmx}

The curve $M_{\PP^2}$ with the two flex lines $F_1,F_2$ gives rise to the four orbifold 
ball-quotient surfaces (previously described by Deraux \cite{Deraux})
on suitable birational transformations of the plane.
We prove that the configuration of curves described in Theorem \ref{thm:uni} is unique up to projective equivalence. 

In \cite{Hirz}, Hirzebruch constructed ball quotient surfaces using arrangements of lines and 
performing Kummer coverings.
It is a well-known question 
whether one can construct other ball quotient surfaces
using higher degree curves, the next case being arrangements of conics. 

Let $\varphi$ be the Cremona transformation of the plane centered 
at the three singularities of  $M_{\PP^2}$. The image  
by $\varphi$ of the curves  $M_{\PP^2},\,F_1,\,F_2,\,L_0$ described in Theorem \ref{thm:uni} is a special arrangement of four plane conics.
We remark that by performing birational transforms 
of $\PP^2$ and by taking the images of the $4$ conics, one can obtain the orbifold ball-quotients of \cite{Deraux}.
To our knowledge that gives the first example of orbifold ball quotients obtained from a configuration of conics
(ball quotient orbifolds obtained from a configuration of a conic and three tangent lines 
are studied in \cite{Holzapfeld} and  \cite{Uludag}).
However we do not know whether one can obtain ball quotient surfaces by 
performing Kummer  coverings branched at these conics.

When preparing this paper, we observed that the mirror $M_{\PP^1 \times \PP^1}$ and one related
 orbifold ball quotient surface among the four might be invariant by an order $2$ automorphism.
Using the equation we have obtained for $M_{\PP^2}$, we prove that this is actually the case: 
there is an involution $\s$ 
on $\PP^1 \times \PP^1$  with fixed point set a $(1,1)$-curve $D_i$ such that the quotient surface is $\PP^2$, 
moreover the image of $D_i$ is a conic $C_o$ and the image of $M_{\PP^2}$ is the unique cuspidal cubic curve $C_u$. 
In the last section we obtain and describe the following result:

\begin{thmx} \label{thm:trois}
There is an orbifold ball-quotient structure on a surface $W$ birational to $\PP^2$ such that the strict transforms  on $W$ of $C_o,C_u$ have weights $2,\infty$ respectively. 
\end{thmx}

The paper is structured as follows:

In section 2, we recall some results of Deraux on the quotient surface $A/G_{48}$ 
and introduce some notation.
In section 3, we study properties of the surface $\PP(1,3,8)$.
In section 4, we introduce the Bolza curve $\t$ and prove that $A/G_{48}$ is isomorphic to $\PP(1,3,8)$.
Section 5 is devoted to the equation of the mirror $M_{\PP^2}$. 
Moreover we describe the four conics configuration.
Section 6 deals with Theorem \ref{thm:trois}.

Some of the proofs in sections 5 and 6 use the computational algebra system Magma, version V2.24-5. A text file containing only the Magma code that appear below is available as an auxiliary file on arXiv and at~\cite{rito}.

Along this paper we use intersection theory on normal surfaces as defined 
by Mumford in \cite[Section 2]{Mumford}.

\textbf{ Acknowledgements}
The last author thanks Martin Deraux for discussions on the problem of proving the isomorphism of $A/G_{48}$ with $\PP(1,3,8).$ 

The second author was supported by FCT (Portugal) under the project PTDC/ MAT-GEO/2823 /2014,
the fellowship SFRH/BPD/111131/2015 and by CMUP (UID/MAT/00144/ 2019),
which is funded by FCT with national (MCTES) and European structural funds through the programs FEDER,
under the partnership agreement PT2020.

\section{Quotient of $A$ by $G_{48}$ and image of the mirrors}

\subsection{\label{subsec:The-images-ofmirr}Properties of $A/G_{48}$ and image of the mirrors}

In this section, we collect some facts from \cite{Deraux} about the action of the automorphism subgroup $G_{48}$ on the Abelian surface \[
A:=\CC^{2}/(\ZZ[i\sqrt{2}])^{2}.
\]

There exists a group $G_{48}$ of order $48$ acting on $A$ 
which is isomorphic to $GL_{2}(\FF_{3})$ (see \cite[Section 3.1]{Deraux} for generators).  
The action of $G_{48}$ on $A$ has no global fixed points
 (in particular some elements have a non-trivial translation part).
 
The group $G_{48}$ contains $12$ order $2$ reflections, i.e. their linear parts acting on 
the tangent space $T_A\simeq \CC^2$  are complex order $2$ reflections. 
The fix point set of a reflection being usually called a mirror, 
we similarly call the fixed point set of a reflection $\tau$ of $G_{48}$ a {\it mirror}. 
The mirror of such a $\tau$ is an elliptic curve on $A$. 
The group $G_{48}$ acts transitively on the set of the $12$ mirrors whose list can be found in~\cite[Table~1]{Deraux}.

We denote by $M$ the union of the mirrors in $A$ and by $M_{48}$ the image of $M$ 
in the quotient surface $A/G_{48}$.  The curve $M_{48}$ is also called the mirror of $A/G_{48}$.

Except the points on $M$, there are two orbits of points in $A$ with non-trivial isotropy, 
one with isotropy group of order 3 at each point, the other with isotropy group of order 8, see~\cite[Proposition 4.4]{Deraux}. 
Correspondingly, the quotient $A/G_{48}$ has two singular points, which are the images of the two special orbits.
\begin{prop}
The surface $A/G_{48}$ is rational and its singularities  are of type $A_{2}+\frac{1}{8}(1,3)$.
\\
The minimal resolution $p:X_{48}\to A/G_{48}$ of the surface $A/G_{48}$
has invariants $K_{X_{48}}^{2}=5$ and $c_{2}(X_{48})=7$. 
\end{prop}
\begin{proof}
Let us compute the invariants of $X_{48}$. Let $\pi:A\to A/G_{48}$
be the quotient map. One has 
\begin{equation}
\OO_{A}=K_{A}=\pi^{*}K_{A/G_{48}}+M,\label{eq:OA}
\end{equation}
moreover, according to \cite[\textsection 4]{Deraux}, each mirror
$M_{i},\,i=1,...,12$, satisfies $M_{i}M=24$, therefore $M^{2}=288$
and 
\[
(K_{A/G_{48}})^{2}=\frac{1}{48}M^{2}=6.
\]
We observe that $M=\pi^{*}\left(\frac{1}{2}M_{48}\right)$, thus by (\ref{eq:OA}),
one gets $M_{48}=-2K_{A/G_{48}}.$ 

The singularities of the quotient surface $A/G_{48}$ are computed in~\cite[Table~2]{Deraux}. Let $C_{1},C_{2}$ be the two $(-3)$-curves above the singularity $\frac{1}{8}(1,3)$; they are  such that $C_1C_2=1$.
Since the singularity of type $A_2$ is an $ADE$ singularity, we obtain:
\[
K_{X_{48}}=p^{*}K_{A/G_{48}}-\frac{1}{2}(C_{1}+C_{2})
\]
and $(K_{X_{48}})^{2}=5$.

Let $\tau$ be a reflection in $G_{48}$ and let $G$ be the Klein group of order $4$ generated by $\tau$ and 
the involution $[-1]_A\in G_{48}$.   One can check that the quotient surface $A/G$ is rational. 
Being dominated by the rational surface $A/G$, the surface $A/G_{48}$  is also rational. 
Thus the second Chern number is $c_{2}(X_{48})=7$ by Noether's formula.
\end{proof}
The mirror $M_{48}$ (the image of $M$ by the quotient map) does not contain singularities of $A/G_{48}$, moreover:
\begin{lem}
The pull-back $\tilde{M}_{48}$ of the  mirror $M_{48}$ 
by the resolution map $p:X_{48}\to A/G_{48}$ has self-intersection
$24$. Its singular set is
\[
2\frak{a}_{2}+\frak{a}_{3}+\frak{a}_{5},
\]
where $\mathfrak{a}_{k}$ denotes a singularity with local equation $y^{2}-x^{k+1}=0$.
\end{lem}
\begin{proof}
The singularities of $\tilde{M}_{48}=p^{*}M_{48}$ are the same
as the singularities of $M_{48}$ since $M_{48}$ is in the smooth
locus of $A/G_{48}$. For the computation of the singularities of $M_{48}$,
we refer to \cite[Table 3]{Deraux}, and for the self-intersection of 
$\tilde{M}_{48}$ (which is the same as the one of $M_{48}$) to \cite[\textsection 6.2]{Deraux}.

\end{proof}

\section{The weighted projective space $\protect\PP(1,3,8)$.\label{sec:The-weighted-projective}}

Since we aim to prove that the quotient surface $A/G_{48}$  is isomorphic to $\PP(1,3,8)$, 
one first  has to study that weighted projective space: this is the goal of this (technical) section. 
The reader might at first browse through the main 
results and notation and proceed to the next section.

\subsection{The surface $\PP(1,3,8)$ and its minimal resolution}
The weighted projective space $\PP(1,3,8)$ is the quotient of $\PP^{2}$
by the group $\ZZ_{3}\times\ZZ_{8}$ generated by 
\[
\sigma=\left(\begin{array}{ccc}
1 & 0 & 0\\
0 & j & 0\\
0 & 0 & \zeta
\end{array}\right)\in PGL_{3}(\CC),
\]
where $j^{2}+j+1=0$ and $\zeta$ is a primitive $8^{th}$ root of
unity. The fixed point set of the order $24$ element $\s$ is 
\[
p_{1}=(1:0:0),\,p_{2}=(0:1:0),\,p_{3}=(0:0:1).
\]
For $i,j\in\{1,2,3\}$ with $i\neq j$ let $L_{ij}'$ be the line
through $p_{i}$ and $p_{j}$. The fixed point set of an order $3$
element (e.g. $\sigma^{8}$) is $p_{2}$ and the line $L_{13}'$.
The fixed point set of an order $8$ element (e.g. $\sigma^{3}$) and its non-trivial powers
is $p_{3}$ and the line $L_{12}'$. Let $\pi:\PP^{2}\to\PP(1,3,8)$
be the quotient map: $\pi$ is ramified with order $3$ over $L_{13}'$
and with order $8$ over $L_{12}'$. The surface $\PP(1,3,8)$ has two singularities, images of $p_{2}$
and $p_{3}$, which are respectively a cusp $A_{2}$ and a singularity
of type $\frac{\ensuremath{1}}{8}(1,3)$.
We denote by $p:Z\to\PP(1,3,8)$
the minimal desingularization map. 
The singularity of type $\frac{\ensuremath{1}}{8}(1,3)$ is resolved by two
rational curves $C_{1},C_{2}$ with $C_{1}C_{2}=1$, $C_{1}^{2}=C_{2}^{2}=-3$, and the singularity $A_2$ is resolved by two
rational curves $C_{3},C_{4}$ with $C_{3}C_{4}=1$, $C_{3}^{2}=C_{4}^{2}=-2$, (see e.g. \cite[Chapter III]{BHPV}).

\begin{lem}~\label{lem:numP138}
The invariants of the resolution $Z$ are 
\[
K_{Z}^{2}=5,\,c_{2}(Z)=7,\,p_{q}=q=0.
\]
\end{lem}
\begin{proof}
We have:
\[
K_{\PP^{2}}\equiv \pi^{*}K_{\PP(1,3,8)}+2L_{13}'+7L_{12}',
\]
therefore since $K_{\PP^{2}}\equiv -3L$, we obtain $\pi^{*}K_{\PP(1,3,8)}\equiv -12L$
and
\[
(K_{\PP(1,3,8)})^{2}=\frac{(-12L)^{2}}{24}=6.
\]

We have 
$$K_{Z}\equiv p^{*}K_{\PP(1,3,8)}-\sum_{i=1}^4a_iC_{i}
$$
where the $a_i$ are rational numbers.
The divisor $K_{Z}$ must satisfy the adjunction formula i.e. one must have
$C_{i}K_{Z}=-2-C_i ^2$ for $i\in\{1,2,3,4\}$. That gives:
\[
K_{Z}=p^{*}K_{\PP(1,3,8)}-\frac{1}{2}(C_{1}+C_{2})
\]
 and therefore $K_{Z}^{2}=5$. For the Euler number, one may use the formula
in \cite[Lemma 3]{RoulleauQuo}:
\[
e(\PP(1,3,8))=\frac{1}{24}(3+2(2-2)+7(2-2)+23\cdot3)=3.
\]
Thus $e(Z)=e(\PP(1,3,8))-2+3+3=7$. Since $\PP(1,3,8)$ is dominated
by $\PP^{2},$ the surface $Z$ is rational, so that $q=p_{g}=0$.
\end{proof}

\subsection{The branch curves in $\PP(1,3,8)$ and their pullback in the resolution}
Let $L_{ij}$ be the image of the line $L_{ij}'$
on $\PP(1,3,8)$ and let $\bar{L}_{ij}$ be the strict transform of
$L_{ij}$ in $Z$. 
\begin{prop}

\label{prop:The-curve-L23 is a (-1)} We have:
\[
\begin{array}{c}
\bar{L}_{23}^2=-1,\,\ \bar{L}_{23}C_{1}=\bar{L}_{23}C_{3}=1,\,\,\bar{L}_{23}C_{2}=\bar{L}_{23}C_{4}=0,\\[\medskipamount]
\bar{L}_{13}^{2}=0,\,\ \bar{L}_{13}C_{2}=1,\,\,\bar{L}_{13}C_{1}=\bar{L}_{13}C_{3}=\bar{L}_{13}C_{4}=0,\\[\medskipamount]
\bar{L}_{12}^{2}=2, \,\, \bar{L}_{12}C_{4}=1,\,\,\bar{L}_{12}C_{1}=\bar{L}_{12}C_{2}=\bar{L}_{12}C_{3}=0.
\end{array}
\]
\end{prop}
\begin{figure}[h]
\caption{\label{fig:Image-of-theTriangle}Image of the lines $L_{ij}'$ in
the desingularisation of $\protect\PP(1,3,8)$}

\includegraphics[scale=0.9]{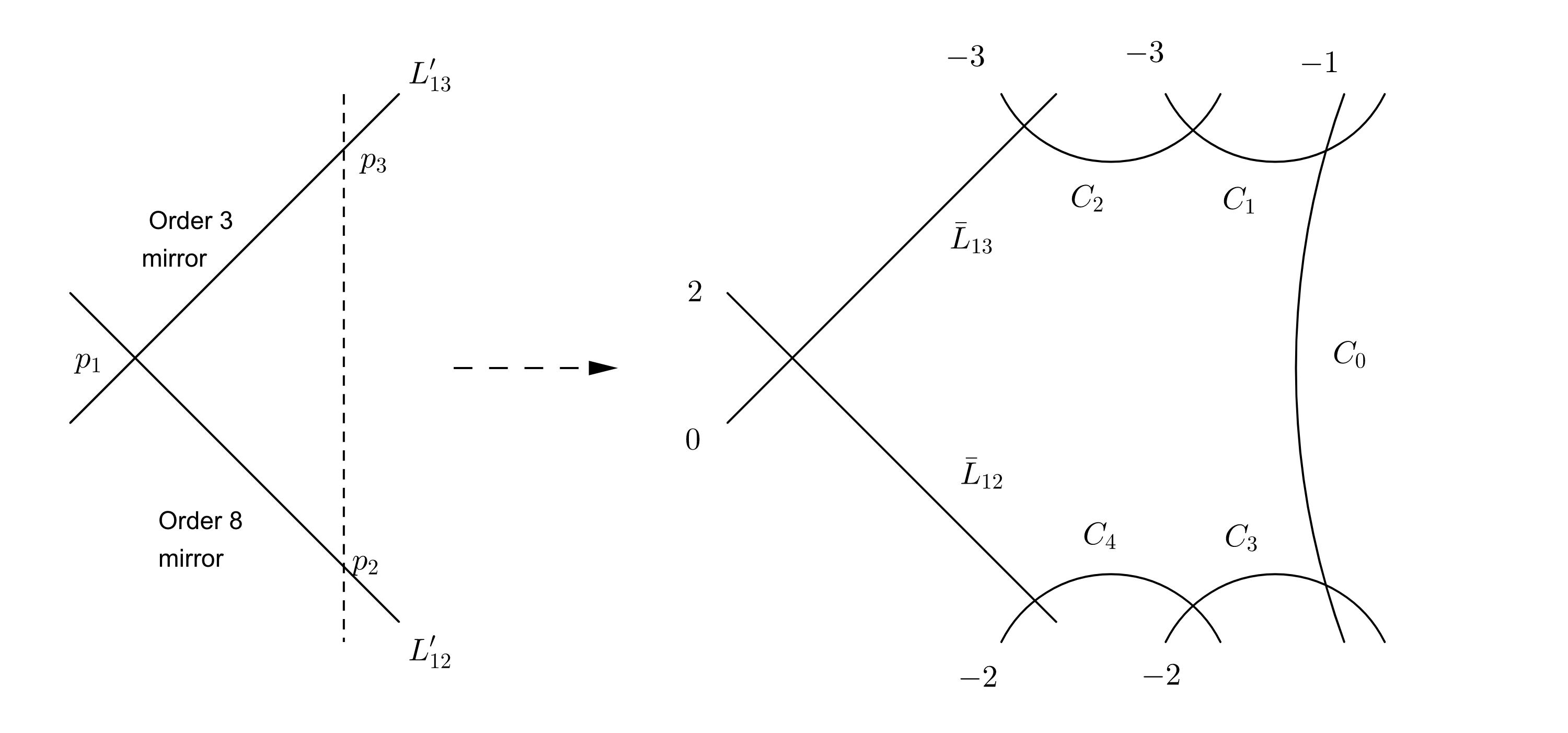}
\end{figure}

\begin{proof}
On $\PP(1,3,8)$ one has $L_{23}^{2}=\frac{1}{24}L_{23}'^{2}=\frac{1}{24}$.
Recall that the resolution map is $p:Z\to\PP(1,3,8)$. Let $a_{1},\dots,a_{4}\in\QQ$
such that 
\[
\bar{L}_{23}=p^{*}L_{23}-\sum_{i=1}^{4}a_{i}C_{i},
\]
then $C_{i}p^{*}L_{23}=0$ for $i\in\{1,2,3,4\}$. Let $u_{i}\in\NN$
such that $C_{i}\bar{L}_{23}=u_{i}$. One gets that 
\[
\begin{array}{cccc}
\left(\begin{array}{c}
a_{1}\\
a_{2}
\end{array}\right)=\dfrac{1}{8}\left(\begin{array}{cc}
3 & 1\\
1 & 3
\end{array}\right)\left(\begin{array}{c}
u_{1}\\
u_{2}
\end{array}\right) &  &  & \left(\begin{array}{c}
a_{3}\\
a_{4}
\end{array}\right)=\dfrac{1}{3}\left(\begin{array}{cc}
2 & 1\\
1 & 2
\end{array}\right)\left(\begin{array}{c}
u_{3}\\
u_{4}
\end{array}\right)\end{array}.
\]
We have $\pi^{*}K_{\PP(1,3,8)}=-12L'_{23}$, thus
\[
K_{\PP(1,3,8)}L_{23}=\frac{1}{24}(-12L_{23}'\cdot L_{23}')=-\frac{1}{2}.
\]
Since $K_{Z}=p^{*}K_{\PP(1,3,8)}-\frac{1}{2}(C_{1}+C_{2})$, we get
\[
\begin{array}{c}
K_{Z}\bar{L}_{23}=\left(p^{*}K_{\PP(1,3,8)}-\frac{1}{2}(C_{1}+C_{2})\right)\left(p^{*}L-\sum_{i=1}^{4}a_{i}C_{i}\right)\\
=-\frac{1}{2}-a_{1}-a_{2}=-\frac{1}{2}(1+u_{1}+u_{2}),
\end{array}
\]
which is in $\ZZ$, with $u_{1},u_{2}\in\NN$. One computes that
\[
\bar{L}_{23}^{2}=\frac{1}{24}-\frac{1}{8}(3u_{1}^{2}+3u_{2}^{2}+2u_{1}u_{2})-\frac{2}{3}(u_{3}^{2}+u_{3}u_{4}+u_{4}^{2})\in\ZZ_{\leq 0}.
\]
Since $K_{Z}\bar{L}_{23}+\bar{L}_{23}^{2}=-2$, the only possibility
is 
\[
\{u_{1},u_{2}\}=\{0,1\},\,\{u_{3},u_{4}\}=\{0,1\},
\]
which gives the intersection numbers with $\bar{L}_{23}$. 

For the curve $L_{13}$, one has $L_{13}K_{\PP(1,3,8)}=-\frac{3}{2}$
and $L_{13}^{2}=\frac{3}{8}$. Let $u:=\bar{L}_{13}C_{1}\in\NN,\,v:=\bar{L}_{13}C_{2}\in\NN$.
Then one similarly computes that 
\[
\bar{L}_{13}K_{Z}=-\frac{1}{2}(3+u+v)\leq-\frac{3}{2}
\]
and 
\[
\bar{L}_{13}^{2}=\frac{1}{8}(3-3u^{2}-3v^{2}-2uv)\leq\frac{3}{8}.
\]
Therefore $\bar{L}_{13}^{2}+K_{Z}\bar{L}_{13}\leq-\frac{9}{8}$ and
since $\bar{L}_{13}^{2}+K_{Z}\bar{L}_{13}\geq-2$, the only solution
is $\{u,v\}=\{0,1\}$, thus $\bar{L}_{13}^{2}=0$ and $\bar{L}_{13}K_{Z}=-2$.

For the curve $L_{12}$, which does not go through the $\frac{1}{8}(1,3)$
singularity, one has 
\[
\bar{L}_{12}K_{Z}=L_{12}K_{\PP(1,3,8)}=-4
\]
 and $L_{12}^{2}=\frac{8}{3}$. Let $w:=\bar{L}_{12}C_{3},\,t:=\bar{L}_{12}C_{4}$.
Then 
\[
\bar{L}_{12}^{2}=\frac{1}{3}(8-2w^{2}-2t^{2}-2wt)\leq\frac{8}{3}.
\]
Therefore $\bar{L}_{12}^{2}+K_{Z}\bar{L}_{12}\leq-\frac{4}{3}$ and
the only solution is $\{w,t\}=\{0,1\}$, thus $\bar{L}_{12}^{2}=2$.
\end{proof}

\subsection{From $\PP(1,3,8)$ to the Hirzebruch surface $\mathbb{F}_3$ and back}
\label{subsub: from p138 to hiz}

By contracting the $(-1)$-curve $C_{0}:=\bar{L}_{23}$ and then the
other $(-1)$-curves appearing from the configuration $C_{1},\dots,C_{4},\bar{L},$
one gets a rational surface with 
\[
K^{2}=2c_{2}=8
\]
containing (depending on the choice of the $(-1)$-curves 
we contract) a curve which either is a $(-2)$-curve
or a $(-3)$-curve. Thus that surface is one of the Hirzebruch surfaces $\FF_{2}$ or $\FF_{3}$.
 Conversely
one can reverse the process and obtain the surface $\PP(1,3,8)$ by
performing a sequence of blow-ups and blow-downs. This process is
unique: this follows from the fact that the automorphism group of a Hirzebruch
surface $\FF_{n},\,n\geq1$ has two orbits, which are the unique $(-n)$-curve
and its open complement (see e.g. \cite{Blanc}). In the sequel, only the connection between $\PP(1,3,8)$ and $\mathbb{F}_3$ will be used.

\section{The Bolza genus $2$ curve in $A$ and its image by the quotient
map \label{sec:section4}} 
In this section we prove that $A/G_{48}$ is isomorphic to $\PP(1,3,8)$.

Let us consider the genus $2$ curve $\t$ whose affine model is 
\begin{equation}
y^{2}=x^{5}-x.\label{eq:affine model}
\end{equation}
It was proved by Bolza \cite{Bol} that the automorphism group of $\t$ is $GL_{2}(\FF_{3})\simeq G_{48}$
and $\t$ is the unique genus $2$ curve with such an automorphism
group.

The automorphisms of $\t$ are generated by the hyperelliptic involution $\l$
and the lift of the automorphism group $G$ of $\PP^{1}$ that preserves
the set of $6$ branch points $0,\,\infty,\,\pm1,\,\pm i$ of the
canonical map $\t\to\PP^{1}$ (i.e. the set of points which are fixed by $\l$). 
Note that actually, any map of degree 2 from $\theta$ to $\PP^1$ is the composition of this map with an automorphism of $\PP^1$. This is a consequence of the two following facts: on the one hand the 6 ramification points (by the Riemann-Hurwitz formula) of such a map are Weierstrass points, and on the other hand the genus 2 curve $\t$ has exactly 6 Weierstrass points.

By the universal property of the Abel-Jacobi map, the group $GL_{2}(\FF_{3})$
acts naturally on the 
Jacobian variety $J(\t)$ of $\t$, 
the action on $\t$ and 
$J(\t)$ being equivariant. 

There is only one Abelian surface
with an action of $GL_{2}(\FF_{3})$, which is $A=E\times E$, where
$E=\CC/\ZZ[i\sqrt{2}]$ as above (see Fujiki \cite{Fujiki} or \cite{BirLange}). We identify $J(\t)$ with $A$.
There are up to conjugation only two possible actions
of $GL_{2}(\FF_{3})$ on $A$ (see \cite{Popov}):\\
a) The action of $G_{48}\simeq GL_{2}(\FF_{3})$ which is described
in sub-section \ref{subsec:The-images-ofmirr}; it has no global
fixed points; \\
b) The one obtained by forgetting the translation part of that action.
That second action globally fixes the $0$ point in $A$. 

Let $\a:\t\hookrightarrow J(\t)=A$ be the embedding of $\t$ sending the point
at infinity of the affine model (\ref{eq:affine model}) to $0$;
we identify $\t$ with its image. 

 Note that the morphism $\theta\times\theta\rightarrow A$, $(x,y)\mapsto [y]-[x]\in{\rm Div}_0(\theta)\simeq A$ is onto since  $\theta\times\theta$ and $A$ are both two-dimensional. Actually, this map has generic degree 2 and contracts the diagonal. Indeed, assume that $[y]-[x]=[y']-[x']$ i.e. $[y]+[x']-[x]-[y']=0\in{\rm Div}_0(\theta)$. If $y'=y$ then $x'=x$ (and conversely) because there is no degree 1 map from $\t$ to $\PP^1$. In the same way, $y=x$ iff $y'=x'$. In the remaining cases, 
there exists a function of degree 2 from $\t$ to $\PP^1$ whose zeroes are $y$ and $x'$ and poles are $x$ and $y'$. But by the remark above, we must have $x'=\l(y)$ and $y'=\l(x)$. Conversely, by the same argument, it is clear that for all $x$ and $y$ in $\t$, $[\l(y)]-[\l(x)]=[x]-[y]$. 

This also implies that the points of the type $[y]-[x]$ with $x$ and $y$ being distinct Weierstrass points are exactly the 2-torsion points of $A$. Indeed, since there are 6 Weierstrass points on $\t$, we have 15 points of that type in $A$ satisfying $[y]-[x]=[\l(x)]-[\l(y)]=[x]-[y]$ i.e. they are 2-torsion points.

The induced linear action b) is given by $g([y]-[x])=[g(y)]-[g(x)]$ for which $0\in{\rm Div}_0(\theta)$ is a fixed point.

If we fix the base point $\infty\in\theta$ then for each $y\in\theta$, $\alpha(x)=[x]-[\infty]$. The induced action of $g\in\aut(\t)$ on $A$ is then given by $g([y]-[x])=[g(y)]-[g(x)]+[g(\infty)]-[\infty]$. This is indeed the only action of $\aut(\t)$ on $A$ commuting with $\alpha$.

\begin{lem}
\label{lem:The-action-of}The action of $GL_{2}(\FF_{3})$ on $A$
inducing the action of $\aut(\t)$ on the curve $\t\hookrightarrow A$
has no global fixed points. 
\end{lem}

\begin{proof}
The fixed points on $A$ for the action of the hyperelliptic involution $\l$ are its points of 2-torsion (and 0). Indeed, $\l([y]-[x])=[\l(y)]-[\l(x)]\in {\rm Div}_0(\theta)$ since $\infty\in\t$ is fixed by $\l$ and, as a consequence of the discussion above, if $[y]-[x]=[\l(y)]-[\l(x)]$ then either $y=x$ or $y=\l(x)$ i.e. $[y]-[x]=[x]-[y]$ and we saw that this implies that $x$ and $y$ are Weierstrass points. 

But for any pair $(x,y)$ of distinct Weierstrass points, it is easy to find $g\in\aut(\t)$ (lifting an automorphism of $\PP^1$) such that $g(\infty)=\infty$ but $[g(y)]-[g(x)]\not = [y]-[x]$.
\end{proof}

For $t\in A$, let $\t_{t}$ be the curve $\t_{t}=t+\t$.
The previous result does not depend on the choice of the embedding $\t\hookrightarrow A$:
indeed the group of automorphisms acting on $A$ and preserving $\t_{t}$
is conjugated by the translation $x\mapsto x+t$ to the group of automorphisms
acting on $A$ and preserving $\t$.

We denote by $H_{48}$ the order $48$ group
acting on $A$ and inducing the automorphism group of the curve $\t\hookrightarrow A$
by restriction.  
As a consequence of Lemma \ref{lem:The-action-of}, we get:
\begin{cor}
There exists an isomorphism between $H_{48}$ and $G_{48}$. That
isomorphism is induced by an automorphism $g$ of the surface $A$
such that $H_{48}=gG_{48}g^{-1}$.
\end{cor}
By \cite[Theorem (0.3)]{Boxall}, the embedding $\a:\t\hookrightarrow A$ is such
that the torsion points of $A$ contained in $\t$ are $16$ torsion
points of order $6$, $5$ torsion points of order $2$ and the origin,
moreover the $x$-coordinates of the $22$ torsion points on $\t$
satisfy 
\[
\begin{array}{cc}
\,\,\,x^{4}-4ix^{2}-1=0,\,\, x^{4}+4ix^{2}-1=0\\
x^{5}-x=0,\,\,\, x=\infty.
\end{array}
\]

\begin{prop}
(a) These $22$ torsion points of $\theta$ are not in the mirror of any of the $12$ complex reflections of
$H_{48}$;\\
(b)  Each of these $22$ points has a
non-trivial stabilizer. 

\end{prop}
\begin{proof}
Let us prove part $(a)$. 

The hyperelliptic involution is given by $(x,y)\to(x,-y)$. By \cite{Cardona},
the rational map 
\[
v:(x,y)\mapsto\Bigl(-\frac{x+i}{ix+1},\sqrt{2}\frac{i-1}{(ix+1)^{3}}y\Bigr)
\]
defines a non-hyperelliptic involution $v$ on $\t$. The $x$-coordinates
of the fixed point set of $v$ are $x_{\pm}=i(1\pm\sqrt{2})$. These
coordinates $x_{\pm}$ are not among the $x$-coordinates of the $22$
torsion points in $\t$. Let $\mathbf{v}$ be the automorphism of
$A$ induced by $v$. The fixed point set of $\mathbf{v}$ is a smooth
genus $1$ curve $E_{v}$ (a mirror) and we have just proved that $E_{v}$
contains no torsion points of $\t$. By transitivity of the group
$H_{48}$ on its set of $12$ non-hyperelliptic involutions, one gets that
no mirror contains any of the $22$ torsion points. 

Let us prove part $(b)$. 

The six $2$-torsion points are the Weierstrass points of the curve
$\theta$, they are fixed by the hyperelliptic involution (whose action
on $A$ has only $16$ fixed points).

The transformation 
\[
w:(x,y)\mapsto \Bigl(\frac{(1+i)x-(1+i)}{(1-i)x+(1-i)},\,-\frac{1}{((1-i)x+(1-i))^{3}}y\Bigr)
\]
defines an order $3$ automorphism of $\t$, which acts symplectically
on $A$ and one computes that it fixes a torsion point $p_{0}=(x_{0},y_{0})$ on $\t$ 
with $x_{0}$ such that $x_{0}^{4}+4ix_{0}^{2}-1=0$, i.e. it is an
order $6$ torsion point. This torsion point is an isolated fixed
point for each non-trivial element of its stabilizer (since by part (a), it is not on a mirror).

Recall that by \cite[Table 2]{Deraux}, there are exactly two orbits of points of respective orders $6$ and $16$
with non-trivial stabilizer under $G_{48}$ which are isolated fixed points of the non-trivial elements of their
stabilizer (by a direct computation one can check that these two orbits are $16$ points of order $6$ and $6$
points of order $2$).
 Since $H_{48}$ is conjugate to $G_{48}$,
the $15$ other $6$-torsion points on $\t$ are also isolated fixed
points for each non-trivial element of their stabilizer.
\end{proof}
Since one can change the embedding  $\t\hookrightarrow A$ by composing
with the automorphism $g$ such that $H_{48}=gG_{48}g^{-1}$, let
us identify $H_{48}$ with $G_{48}$. 

By sub-section \ref{subsec:The-images-ofmirr} (or \cite{Deraux}),
the images of the $22$ torsion points of $\t$ on the quotient surface
$A/G_{48}$ give the singularities $A_{2}$ and $\frac{1}{8}(1,3)$. 

Let $m$ be the mirror of one of the $12$ complex reflections in $G_{48}$.
\begin{lem}~\label{lemma:intersectionbolzamirror}
One has $\t\cdot m=2$.
\end{lem}
\begin{proof}
The intersection number $\t\cdot m$ is the number of fixed points
of the involution $\iota_{m}$ with mirror $m$ restricted to $\t$.
Since $\iota_{m}$ fixes exactly one holomorphic form, the quotient
of $\t$ by $\iota_{m}$ is an elliptic curve, thus by the Hurwitz formula
$\t\cdot m=2$.
\end{proof}
Let $\t_{48}$ be the image of $\t$ in $A/G_{48}$. One has:
\begin{prop} \label{M48etC0}
The strict transform $C_{0}$ of $\t_{48}$ by the resolution $X_{48}\to A/G_{48}$
is a $(-1)$-curve and we have $\tilde{M}_{48}C_0=1$.
\end{prop}
\begin{proof}
One has 
\[
\t_{48}^{2}=\frac{1}{48}\t^{2}=\frac{1}{24}.
\]
Let $\pi:A\to A/G_{48}$ be the quotient map; it is ramified with
order $2$ on the union $M$ of the $12$ mirrors. One has $\pi^{*}(K_{A/G_{48}}+\frac{1}{2}M_{48})=K_{A}=0$,
thus
\[
K_{A/G_{48}}\t_{48}=-\frac{1}{48}(M\t)=-\frac{1}{48}12\cdot2=-\frac{1}{2}.
\]
The curve $\t_{48}$ contains the singularities $\frac{1}{8}(1,3)$
and $A_{2}$ (image respectively of the $2$-torsion points and the
$6$-torsion points of $\t$). We are then left with the same combinatorial
situation as in the computation of $\bar{L}_{23}^{2}$ in Proposition
\ref{prop:The-curve-L23 is a (-1)}, thus we conclude that $C_{0}^{2}=-1$.

The two intersection
points of $m$ and $\theta$ in Lemma~\ref{lemma:intersectionbolzamirror} are permuted by the hyperelliptic involution
of $\theta$ thus $M_{48}\theta_{48}=1$, which implies $\tilde{M}_{48}C_0=1$.
\end{proof}

We obtain:
\begin{thm}
The surface $A/G_{48}$ is isomorphic to $\PP(1,3,8)$.
\end{thm}
\begin{proof}
Let us denote the resolution map by $p:X_{48}\to A/G_{48}$. Let
$C_{1},C_{2}$ be the resolution curves of the singularity $\frac{1}{8}(1,3)$,
and $C_{3},C_{4}$ be the resolution of $A_{2}$. Let $a\in A$ be an isolated
fixed point of an automorphism $\tau$ of order $3$ or $8$. The tangent
space $T_{\t,a}\subset T_{A,a}$ is stable by the action of $\tau$.
Since the local setup is the same, we can reason as in Proposition \ref{prop:The-curve-L23 is a (-1)}
and we obtain that the curve $C_{0}$ is such that 
\[
C_{0}C_{1}=C_{0}C_{3}=1,\,\,C_{0}C_{2}=C_{0}C_{4}=0.
\]
Contracting the curves $C_{0},C_{1},C_{2}$, one gets a rational surface
with a $(-3)$-curve and with invariants $K^{2}=2c_{2}=8$. This is
therefore the Hirzebruch surface $\FF_{3}$. From section \ref{sec:The-weighted-projective},
we know that reversing the contraction process one gets the weighted
projective plane $\PP(1,3,8)$ (contracting the curves $C_{0},C_{1},C_{3}$, one would have obtained the Hirzebruch surface $\FF_{2}$).
\end{proof}

\begin{rem}\label{rem:identifBaseNS}
Now we identify $\PP(1,3,8)$ with $A/{G_{48}}$ and we use the
notation in section \ref{sec:The-weighted-projective}. In particular
$Z=X_{48}$ is the minimal resolution of $\PP(1,3,8)$, the curves
$C_{1},\dots,C_{4}$ are exceptional divisors of the resolution map $Z\to \PP(1,3,8)$ and $C_{0}=\bar{L}_{23}$
is a $(-1)$-curve in $Z$.

Let us observe that the divisor $\tilde{F}=C_{1}+3C_{0}+2C_{3}+C_{4}$ satisfies 
\[
\tilde{F}C_{1}=\tilde{F}C_{0}=\tilde{F}C_{3}=\tilde{F}C_{4}=0,
\]
thus $\tilde{F}^{2}=0$, moreover $\tilde{F}C_{2}=\bar{L}_{13}\tilde{F}=1$,
$\tilde{F}\bar{L}_{13}=0$ and $\bar{L}_{13}^{2}=0$. This implies
that the curves $\tilde{F}$ and $\bar{L}_{13}$ are fibers of the
same fibration onto $\PP^{1}$ and $C_{2}$ is a section of that fibration.

The curves $C_0,\dots,C_4$  are exceptional divisors or strict transform of generators of the 
N\'eron-Severi group of a minimal rational surface. 
Thus the N\'eron-Severi group of the rational surface $X_{48}$ is generated by these curves. 
Knowing the intersection of curves $\bar L _{12},\,\bar L _{13},\,\tilde M_{48}$ with 
these curves (see Propositions \ref{prop:The-curve-L23 is a (-1)} and \ref{M48etC0}) 
 it is easy to obtain their classes in the N\'eron-Severi group,  in particular one gets that
  $\bar L _{12} \tilde M_{48}=8,\, \bar L _{13} \tilde M_{48}=3.$
\end{rem} 

\begin{figure}[h]

\caption{\label{fig:The-Yoga:-image} Configuration of curves $\tilde M_{48}$ , $\bar L _{12}$, $\bar L _{13}$ etc... in $X_{48}$ and their intersection numbers}
\includegraphics[scale=0.30]{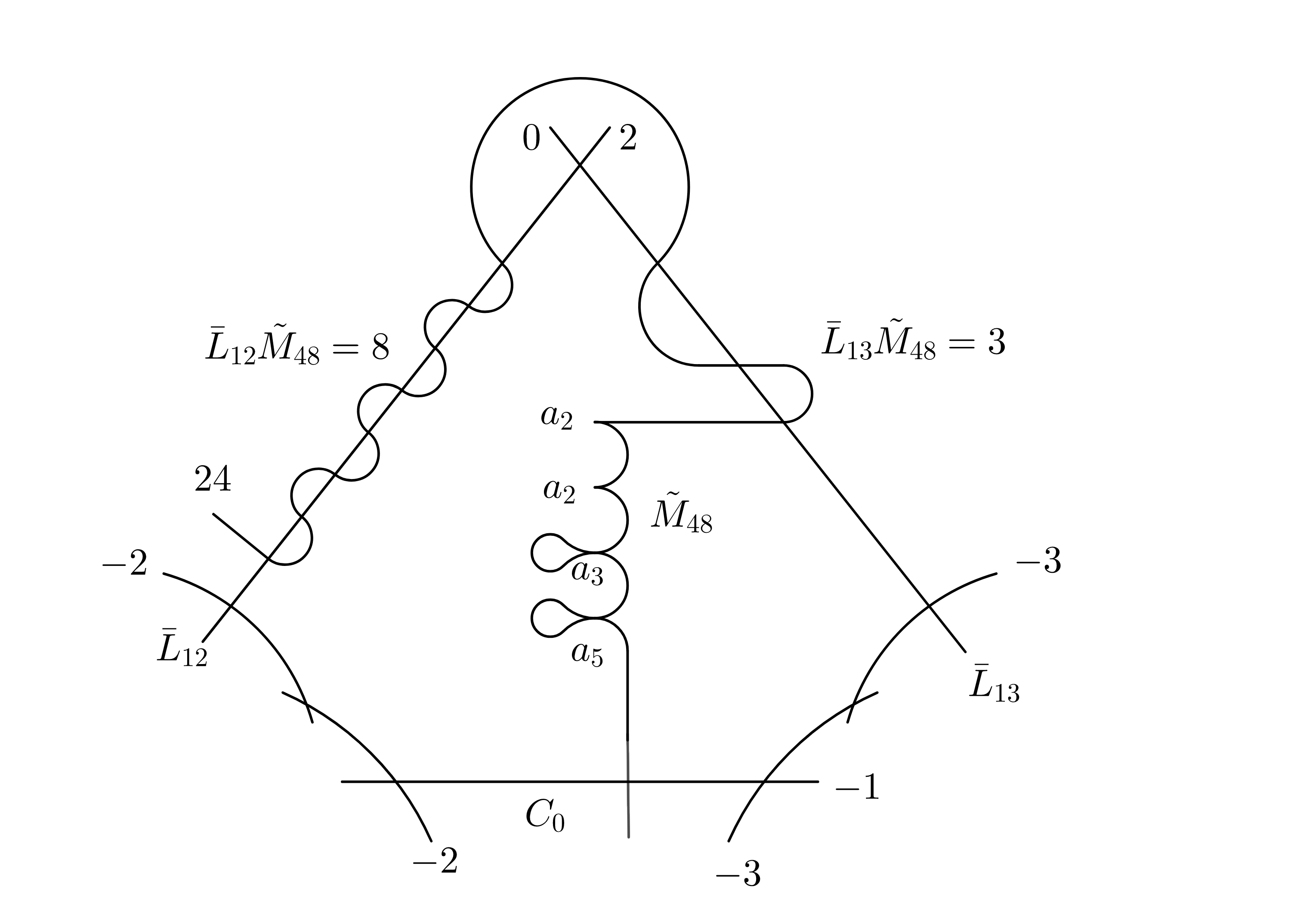}

\end{figure}

\section{A model of the mirror}

\subsection{A birational map from $\protect\PP(1,3,8)$ to $\protect\PP^{1}\times\protect\PP^{1}$ 
; images of the mirror}

\subsubsection{A rational map $\protect\PP(1,3,8)\dashrightarrow \protect\PP^{1}\times\protect\PP^{1}$.\label{subsec:A-rational-map}}

As above, we identify $\PP(1,3,8)$ with $A/{G_{48}}$; we use the
notation of sections \ref{sec:The-weighted-projective} and \ref{sec:section4}.

Take a point $p$ in the Hirzebruch surface $\mathbb F_n$ that is not in the negative section.
By blowing-up at $p,$ and then by  blowing-down  the strict transform of the fiber through $p,$
we get the Hirzebruch surface $\mathbb F_{n-1}.$ This process is called an {\em elementary transformation}.

Recall from sections \ref{sec:The-weighted-projective} and \ref{sec:section4} that there is a map $\psi:\mathbb P(1,3,8)\dashrightarrow\mathbb F_3$ that contracts
the curves $C_0,C_3,C_4$ to a smooth point.


Performing any sequence of three elementary transformations as above, we get a map $\rho:\mathbb F_3\dashrightarrow\mathbb F_0=\mathbb P^1\times\mathbb P^1$. This can be seen as a birational transform that, by blowing-up three times at a point $q$ not contained in the negative section, takes the fibre $F_q$ through $q$ to a chain of curves with self intersections $(-1),(-2),(-2),(-1),$ then followed by the contraction of the $(-1),(-2),(-2)$ chain (which contains the strict transform of $F_q$). For our purpose, we choose the three points to blow-up in a specific way, see subsection~\ref{subsec:blabla}.

Consider $$\phi:=\rho\circ\psi:\mathbb P(1,3,8)\dashrightarrow\mathbb P^1\times\mathbb P^1.$$

We observe that given any two points $t,t'\in\mathbb P^1\times\mathbb P^1$ not in a common fiber,
the map $\phi$ can be chosen such that the inverse $\phi^{-1}$ is not defined at $t,t'$
and $\phi^{-1}(\mathbb P^1\times\mathbb P^1)=\mathbb P(1,3,8).$

\begin{figure}[h]
\caption{From $X_{48}$ to $\protect\PP^{1}\times\protect\PP^{1}$ \label{fig:yoga}and
back}

\includegraphics[scale=0.25]{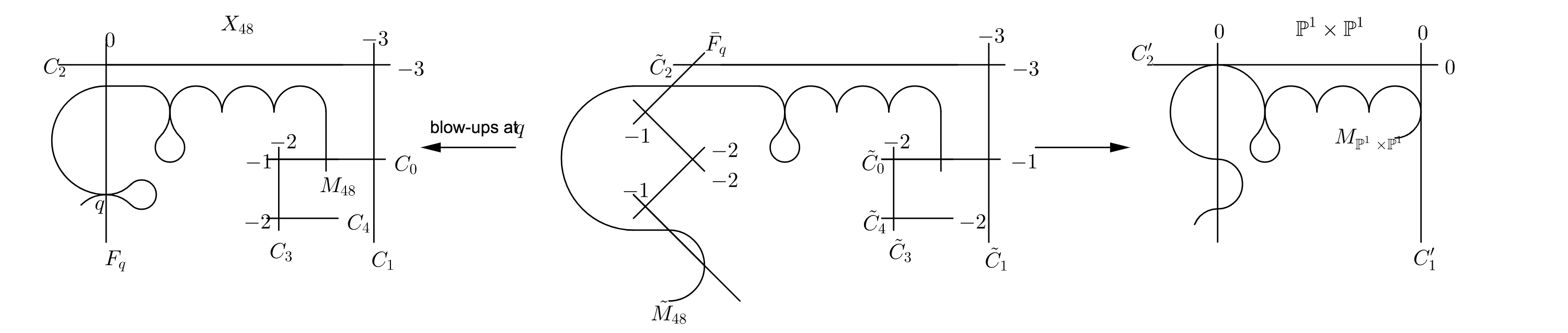}
\end{figure}

\subsubsection{Image of the mirror $M_{48}$ in $\protect\PP^{1}\times\protect\PP^{1}$.} \label{subsec:blabla}

Let us describe how to choose $\phi$ such that the image $M_{\mathbb P^1\times\mathbb P^1}$
of the mirror curve $M_{48}$ is a $(3,3)$-curve with singularities $\mathfrak a_3 + 2\mathfrak a_2$
and two special fibers tangent to it with multiplicity $3.$

The map $\mathbb P(1,3,8)\dashrightarrow\mathbb F_3$ factors through a morphism
$\varphi:X_{48}\rightarrow\mathbb F_3.$
Consider the point $t_0:=\varphi(C_0).$
Since $M_{48}C_0=1,$ then $\varphi(M_{48})$ is a curve which is smooth at
$t_0$ and its intersection number with the curve $\varphi(C_1)$ at $t_0$ is $3$.
The curve $C_1':=\rho\circ\varphi(C_1)$ is a fiber of $\mathbb P^1\times\mathbb P^1.$

Then we choose $q$ to be the $\mathfrak{a}_{5}$-singularity
of $M_{48}$. The fiber $F_{q}$ through $q$ cuts $M_{48}$ at $q$
with multiplicity $2$ or $3$. Suppose that the multiplicity is $3$. Then by taking the blow-up at that point and computing the strict transform of the curves $F_q$ 
and $M_{48}$, one can check that $F_q M_{48}\geq 4$. But $F_q M_{48}=\bar L _{13} M_{48}=3$ by Remark~\ref{rem:identifBaseNS}.
Therefore the fiber $F_{q}$ through $q$ cuts $M_{48}$ at $q$
with multiplicity $2$, and at another point. 
\begin{rem} \label{rem:singa3}
An analogous reasoning gives that the fiber through
 the $\mathfrak{a}_{3}$-singularity has the same property: it is transverse
 to the tangent of the $\mathfrak{a}_{3}$-singularity.
 \end{rem}
The three successive
blow-ups above $q$ are chosen such that they resolve the singularity
$\mathfrak{a}_{5}$.
The three blow-downs we described create a multiplicity $3$ tangent
point between $M_{\PP^{1}\times\PP^{1}}$ (the image of $M_{48}$ in $\PP^{1}\times\PP^{1}$) and the curve 
$C_{2}'$ (the image of $C_2$), thus $C_2 'M_{\PP^{1}\times\PP^{1}}=3$. Moreover $C_2'^2=0,C_1'C_2'=1$ (see figure \ref{fig:yoga}).

The mirror $M_{48}$ does not cut the curves $C_{1}$
and $C_{2}$. The transforms of these curves in $\PP^{1}\times\PP^{1}$
are fibers $C_{1}',C_{2}'$ such that $C_{i}'$ cuts $M_{\PP^1 \times \PP^1}$ at one point only, with multiplicity $3$.  
In particular, the class of $M_{\PP^{1}\times\PP^{1}}$ in the N\'eron-Severi
group of $\PP^{1}\times\PP^{1}$ is $3C_{1}'+3C_{2}'$. 
The singularities of $M_{\PP^{1}\times\PP^{1}}$ are $\mathfrak{a}_{3}+2\mathfrak{a}_{2}$.

\subsubsection{From  $\PP^{1}\times\PP^{1}$ to $\PP^{2}$ and back} \label{subsec:blabla2}
Let us recall that the blowup of $\mathbb P^1\times\mathbb P^1$ at a point,
followed by the blow-down of the strict transform of the two fibers through that point,
gives a birational map $\mathbb P^1\times\mathbb P^1\dashrightarrow\mathbb P^2.$

We choose to blow-up the point at the $\mathfrak{a}_{3}$-singularity
$s_{0}$, so that the strict transform of $M_{\PP^{1}\times\PP^{1}}$
has a node above $s_{0}$. The two fibers $F_{1},F_{2}$ of $\PP^{1}\times\PP^{1}$ passing through $s_{0}$
cut $M_{\PP^{1}\times\PP^{1}}$ in two other points respectively $s_{1},s_{2}$ (see Remark \ref{rem:singa3}; the result is preserved through the birational process).
The fibers $F_{1},F_{2}$ are contracted into points in $\PP^{2}$
by the rational map $\PP^{1}\times\PP^{1}\dashrightarrow\PP^{2}$,
the images of $s_{1},s_{2}$ by that map are on the image of the exceptional
divisor, which is a line $L_0$ through the node.  
This implies that the strict transform of $M_{\PP^1\times \PP^1}$
 is a plane quartic curve $M_{\mathbb P^2}$.
The process in illustrated in Figure \ref{fig:ImageP1P1toP2}. 

\begin{figure}[h]
\caption{ \label{fig:ImageP1P1toP2} From $\PP^1\times \PP^1$ to $\PP^2$}
\includegraphics[scale=0.3]{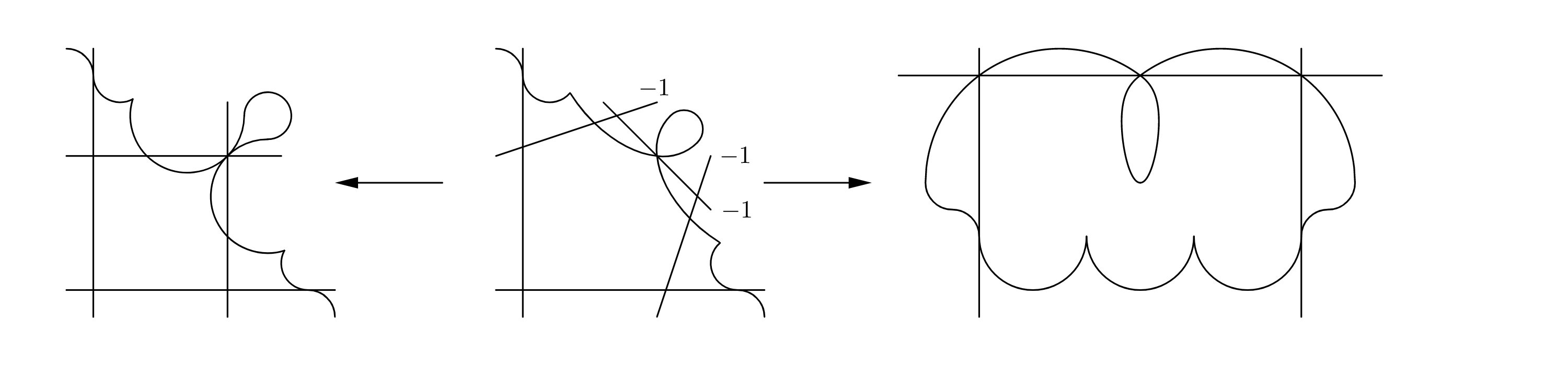}
\end{figure}

The total transform of $M_{\PP^1 \times \PP^1}$ in $\PP^{2}$ is
the union of $2L_{0}$ with $M_{\PP^2}$.  This quartic $M_{\PP^2}$ has the following 
properties which follow from its description  and the choice of the transformation 
from $\PP^{1}\times\PP^{1}$ to $\PP^{2}$:

\begin{prop} \label{prop:SingMP2}
The singular set of the quartic curve $M_{\mathbb P^2}$ is $\mathfrak a_1 + 2\mathfrak a_2,$
and the nodal point is contained in the line $L_0$.
The curve $M_{\mathbb P^2}$ contains two flex points such that each corresponding tangent line
meets the quartic at a second point that is contained in the line $L_0.$
\end{prop}

\subsection{The yoga between the mirrors $M_{\PP^2}$ and $M_{48}$}
Using the previous description the reader can follow the transformations between the surfaces $\PP(1,3,8)$ and the plane. The link between Deraux's ball quotient orbifolds described in \cite[Theorem 5]{Deraux}
 and the quartic $M_{\PP^2}$ is as follows:
 
 The singularities $\mathfrak{a}_1+2\mathfrak{a}_2$ of $M_{\PP^2}$ correspond respectively
to singularities $\mathfrak{a}_3+2\mathfrak{a}_2$ of $M_{48}$, so that in order to get 
the curves $F,G,H$ in \cite[Figure 1]{Deraux} one has to blow-up and contract at these $3$
points as it is done in \cite{Deraux}. 
In order to obtain the curve $E$  in \cite[Figure 1]{Deraux}, one has to blow-up the two flexes 
three times in order to separate $M_{\PP^2}$ and the flex lines. One obtain two chains of $(-1),(-2),(-2)$ curves.
Contracting one of the two $(-2),(-2)$ chains one gets an $A_2$-singularity. 
The curve $E$ is the image by the contraction map of the remaining $(-1)$-curve of the chain.
The resolution of the singularity $A_2$ on $\PP(1,3,8)$  corresponds to the two $(-2)$-curves 
on the other chain of $(-1),(-2),(-2)$ curves. After taking the blow-up at the residual 
intersection of the quartic and the flex lines and after separating the flex lines and the mirror
$M_{\PP^2}$, one gets two $(-3)$-curves intersecting transversally at one point.
In that way the resolution of the singularity $\frac{1}{8}(1,3)$ on $\PP(1,3,8)$ 
by two $(-3)$-curves corresponds to the two flex lines.

\subsection{A particular quartic curve in $\protect\PP^{2}$}

The aim of this sub-section is to prove the following result:
\begin{thm}
\label{thm:unicity of quartic} Up to projective equivalence, there
is a unique quartic curve $Q$ in $\mathbb{P}^{2}$ with distinct
points $p_{1},\ldots,p_{7}$ such that: 

\begin{enumerate} \item $Q$ has a node at $p_{1}$ and ordinary
cusps at $p_{2},\,p_{3}$; 

\item the points $p_{4},\,p_{5}$ are flex points of $Q$; 

\item the tangent lines to $Q$ at $p_{4},\,p_{5}$ contain $p_{6},\,p_{7},$
respectively; 

\item the line through $p_{6},\,p_{7}$ contains $p_{1}.$ 

\end{enumerate}

We can assume that 
\[
p_{1}=[0:0:1],\ p_{2}=[0:1:1],\ p_{3}=[1:0:1].
\]
Then the equation of $Q$ is 
\[
(x^{2}+xy+y^{2}-xz-yz)^{2}-8xy(x+y-z)^{2}=0,
\]
and the points $p_{4},p_5$ and $p_6,p_{7}$ are, respectively, 
\[
\left[\pm2\sqrt{-2}+8:\mp2\sqrt{-2}+8:25\right],\  \left[\pm2\sqrt{-2}:\mp2\sqrt{-2}:1\right].
\]
\end{thm}

\begin{cor}
The mirror $M_{\mathbb{P}^2}$  described on sub-section \ref{subsec:blabla2} 
satisfies the hypothesis of Theorem \ref{thm:unicity of quartic}, thus $M_{\mathbb{P}^2}$ is projectively equivalent to the quartic $Q$.
\end{cor}

\begin{figure}[h]
\caption{ \label{fig:yoga4} The quartic $Q$}
\includegraphics[scale=0.45]{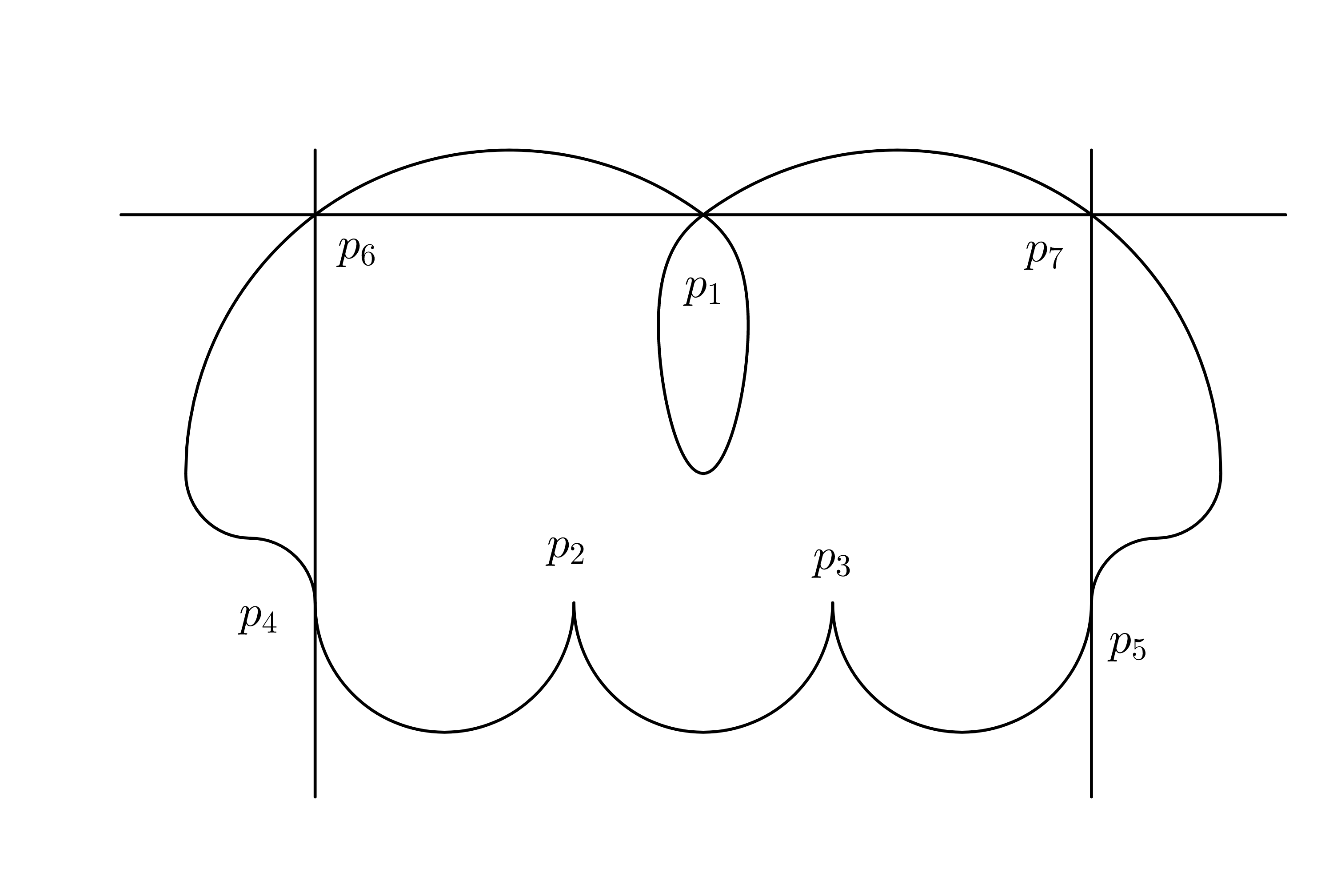}
\end{figure}

In order to prove \ref{thm:unicity of quartic}, let us first give a criterion for the existence of roots of multiplicity
at least $3$ on homogeneous quartic polynomials on two variables.
We use the computational algebra system Magma; see \cite{rito} for a copy-paste ready version of the Magma code.
\begin{lem}
\label{lem:criterion}The polynomial 
\[
P(x,z)=ax^{4}+bx^{3}z+cx^{2}z^{2}+dxz^{3}+ez^{4}
\]
has a root of multiplicity at least $3$ if and only if 
\[
12ae-3bd+c^{2}=27ad^{2}+27b^{2}e-27bcd+8c^{3}=0.
\]
 
\end{lem}
\begin{proof}
The computation below is self-explanatory. 
\begin{verbatim} 
R<u,v,m,n,a,b,c,d,e>:=PolynomialRing(Rationals(),9);
P<x,z>:=PolynomialRing(R,2); 
f:=(u*x+v*z)^3*(m*x+n*z); 
s:=Coefficients(f); 
I:=ideal<R|a-s[5],b-s[4],c-s[3],d-s[2],e-s[1]>; 
EliminationIdeal(I,4);
\end{verbatim}  \end{proof}

Let us now prove Theorem \ref{thm:unicity of quartic}:
\begin{proof}

We have already chosen $3$ points $p_1,p_2,p_3$ in $\PP^2$. 
Instead of choosing a fourth point for having a projective base,
one can fix two infinitely near points over $p_2$ and $p_3$.
Indeed the projective transformations that fix points $p_1,p_2,p_3$ are of the form
\[
\phi:[x:y:z]\longmapsto[ax:by:(a-1)x+(b-1)y+z]
\]
and these transformations act transitively on the lines through $p_2$ and $p_3$. 
 Thus  up to projective equivalence, we can fix the tangent cones 
(which are double lines) of the curve $Q$ at the cusps $p_{2},\,p_{3}.$
Let us choose for these cones the lines
with equations $y=z$ and $x=z$, respectively.

The linear system of quartic curves in $\PP^{2}$ is $14$ dimensional.
The imposition of a node and two ordinary cusps (with given tangent cones) corresponds to $13$
conditions, thus we get a pencil of curves. We compute that this pencil
 is generated by the following quartics:
\[
(x^{2}+xy+y^{2}-xz-yz)^{2}=0,\ \ xy(x+y-z)^{2}=0.
\]
Notice that, at the points $p_2,p_3,$ the first generator is of multiplicity $2$ 
and the second generator is of multiplicity $3$,
 thus a generic element in the pencil has a cusp singularity at $p_2,p_3.$

Let us compute the quartic curves $Q$ satisfying condition (1) to (4) of Theorem \ref{thm:unicity of quartic}. The method is to
define a scheme by imposing the vanishing of certain polynomials $P_{i}=0,$
and the non-vanishing of another ones $D_{i}\neq0,$ which is achieved
by using an auxiliary parameter $n$ and imposing $1+nD_{i}=0.$

\begin{verbatim} 
K:=Rationals();
R<a,q1,q2,m,d1,d2,n>:=PolynomialRing(K,7); 
P<x,y,z>:=ProjectiveSpace(R,2); 
\end{verbatim}
The defining polynomial of $Q,$ depending on one parameter: 

\begin{verbatim}
F:=(x^2 + x*y + y^2 - x*z - y*z)^2 + a*x*y*(x + y - z)^2; 
\end{verbatim}
The points $p_{6},\,p_{7}$ are in a line $y=mx,$ hence they are
of the form 

\begin{verbatim} 
p6:=[q1,m*q1,1]; 
p7:=[q2,m*q2,1];
\end{verbatim} 
and we must have the vanishing of 

\begin{verbatim} 
P1:=Evaluate(F,[q1,m*q1,1]); 
P2:=Evaluate(F,[q2,m*q2,1]); 
\end{verbatim} 
The defining polynomials of lines through that points are:

\begin{verbatim} 
L1:=-y+d1*x+(m*q1-d1*q1)*z; 
L2:=-y+d2*x+(m*q2-d2*q2)*z; 
\end{verbatim} 
We need to impose that these lines are not tangent to $Q$ at $p_{6},\,p_{7},$
thus the following matrices must be of rank $2.$

\begin{verbatim} 
M1:=Matrix([JacobianSequence(F),JacobianSequence(L1)]); 
M1:=Evaluate(M1,[q1,m*q1,1]);
M2:=Matrix([JacobianSequence(F),JacobianSequence(L2)]);
M2:=Evaluate(M2,[q2,m*q2,1]); 
\end{verbatim} 
The matrix $M_{i}$ is of rank $2$ if one of its minors is non-zero.
Here we make a choice for these minors, but in order to cover all
cases the computations must be repeated for all other choices. 

\begin{verbatim} 
D1:=Minors(M1,2)[1]; 
D2:=Minors(M2,2)[1];
\end{verbatim} 
Now we intersect the quartic $Q$ with the lines $L_{1},\,L_{2}:$ 

\begin{verbatim} 
R1:=Evaluate(F,y,d1*x+(m*q1-d1*q1)*z); 
R2:=Evaluate(F,y,d2*x+(m*q2-d2*q2)*z); 
\end{verbatim} 
and we use Lemma \ref{lem:criterion} to impose that these lines are
tangent to $Q$ at flex points of $Q$: 

\begin{verbatim} 
c:=Coefficients(R1); 
P3:=c[1]*c[5]-1/4*c[2]*c[4]+1/12*c[3]^2; 
P4:=c[1]*c[4]^2+c[2]^2*c[5]-c[2]*c[3]*c[4]+8/27*c[3]^3;
c:=Coefficients(R2);
P5:=c[1]*c[5]-1/4*c[2]*c[4]+1/12*c[3]^2; 
P6:=c[1]*c[4]^2+c[2]^2*c[5]-c[2]*c[3]*c[4]+8/27*c[3]^3;
\end{verbatim} 
We note that the lines $L_{1},\,L_{2}$ cannot contain the points
$p_{2},\,p_{3}:$

\begin{verbatim}
D3:=Evaluate(L1,[0,1,1]); 
D4:=Evaluate(L1,[1,0,1]); 
D5:=Evaluate(L2,[0,1,1]); 
D6:=Evaluate(L2,[1,0,1]); 
\end{verbatim}
Also the line $L_{i}$ cannot contain the point $p_{1},$ $i=1,2:$

\begin{verbatim} 
D7:=(m-d1)*(m-d2); 
\end{verbatim} 
And it is clear that the following must be non-zero: 
\begin{verbatim}
D8:=a*q1*q2*(q1-q2);
 \end{verbatim}
Finally we define a scheme with all these conditions. 

\begin{verbatim} 
A:=AffineSpace(R); 
S:=Scheme(A,[P1,P2,P3,P4,P5,P6,1+n*D1*D2*D3*D4*D5*D6*D7*D8]); 
\end{verbatim} 
We compute (that takes a few hours):
\begin{verbatim}
PrimeComponents(S); 
\end{verbatim}  
and get the unique solution $a=-8.$
\end{proof}

From the equation of the quartic $Q=M_{\PP^2}$, one can compute a degree $24$ equation for the mirror 
$M_{48}$, which is: 

\footnotesize
\begin{verbatim}
(31072410*r+44060139)*x^24+(599304420*r-4660302600)*x^21*y+(-106415505000*r+18054913500)*x^18
*y^2+(796474485000*r+3638808225000)*x^15*y^3+(-27123660*r-18697014)*x^16*z+(34521715125000
*r-31210968093750)*x^12*y^4+(107726220*r+2948918400)*x^13*y*z+(-257483985484500*r-
516632817969000)*x^9*y^5+(42798843000*r-32351244300)*x^10*y^2*z+(-1747212737190000*r
+3228789525752500)*x^6*y^6+(-407331396000*r-935091495000)*x^7*y^3*z+(-655139025450000*r+
10855982580975000)*x^3*y^7+(7724970*r-2222037)*x^8*z^2+(-3383703150000*r+9052448883750)
*x^4*y^4*z+(1544666220033750*r+11942493993804375)*y^8+(-102498120*r-465161400)*x^5*y*z^2+
(-319463676000*r+12613760073000)*x*y^5*z+(-2705586000*r+7086771600)*x^2*y^2*z^2+(-712080*r
+1186268)*z^3=0
\end{verbatim}
\normalsize
where $r=\sqrt{-2}$.

\subsection{A configuration of four plane conics related to the orbifold ball quotient}
In this subsection we describe the configuration of conics which we announced in the introduction.

Let us consider a conic tangent to two lines of a triangle in $\PP^2$, and going
through two points of the remaining line. Performing a Cremona transformation
at the three vertices of the triangle one obtains a quartic curve in
$\PP^{2}$ with singularities $\frak{a}_{1}+2\frak{a}_{2}$. Conversely,
starting with such a quartic, its image by the Cremona transform at
the three singularities is a conic with three lines having
the above configuration. 

Thus we consider the Cremona transform $\varphi$ at the three singularities 
of the quartic $M_{\PP^2}$. Let  $D_{1},\dots,D_{4}$  be respectively the images of $M_{\PP^2}$, 
the line $L_0$ through the node and the two residual points of the flex lines,
and the two flex lines. Using Magma, we see that these are $4$ conics meeting in $10$ points, as follows:
\begin{center}
\begin{tabular}{|c|c|c|c|c|c|c|c|c|c|c|}
\hline 
 & $q_{1}$ & $q_{2}$ & $q_{3}$ & $q_{4}$ & $q_{5}$ & $q_{6}$ & $q_{7}$ & $q_{8}$ & $q_{9}$ & $q_{10}$\tabularnewline
\hline 
$D_{1}$ & 1+ & 1+ & 0 & 0 & 0 & 0 & 1 & 1 & 1 & 1\tabularnewline
\hline 
$D_{2}$ & 1 & 1 & 1 & 1 & 1 & 0 & 0 & 0 & 1 & 1\tabularnewline
\hline 
$D_{3}$ & 0 & 1+ & 1 & 1 & 1 & 1 & 1 & 0 & 0 & 0\tabularnewline
\hline 
$D_{4}$ & 1+ & 0 & 1 & 1 & 1 & 1 & 0 & 1 & 0 & 0\tabularnewline
\hline 
\end{tabular} 
\end{center}
\vspace{1mm}
Here two $+$ in the column of $q_{j}$ mean that the two curves meet with
multiplicity $3$ at point $q_{i}$. The other intersections are transverse.
We see that the various ball-quotient orbifolds that Deraux described in \cite{Deraux} may be obtained from a configuration of conics by performing birational transformations.

\section{One further quotient by an involution }

\subsection{The quotient morphism $\PP^1\times \PP^1\to \PP^2$, image of the mirror as the cuspidal cubic}
Consider the plane quartic curve $Q$ from Theorem \ref{thm:unicity of quartic}. Here we show the existence of a birational map
$$\rho:\mathbb P^2\dashrightarrow\mathbb P^1\times\mathbb P^1\subset\mathbb P^3$$
and an involution $\sigma$ on $\mathbb P^1\times\mathbb P^1$ that preserves $\rho(Q)$ and fixes the diagonal $D$ of
$\mathbb P^1\times\mathbb P^1$ pointwise.
Moreover, we have $\left( \mathbb P^1\times \mathbb P^1\right)/\sigma=\mathbb P^2,$ and
the images $C_u,\, C_o$ of $\rho(Q),\, D$ are curves of degrees $3,2,$ respectively.
The curve $C_u$ has a cusp singularity and intersects $C_o$ at three points, with intersection multiplicities $4,1,1.$ The map $\rho$ is the inverse of the birational transform $\PP^1\times \PP^1\dashrightarrow \PP^2$ described in sub-section~\ref{subsec:blabla2}, whose indeterminacy is at the singularity $\mathfrak{a}_3$ of $M_{\PP^1\times \PP^1}.$

\begin{verbatim}
K:=Rationals();
R<r>:=PolynomialRing(K);
K<r>:=ext<K|r^2+2>;
P2<x,y,z>:=ProjectiveSpace(K,2);
Q:=Curve(P2,(x^2+x*y+y^2-x*z-y*z)^2-8*x*y*(x+y-z)^2);
p6:=P2![2*r,-2*r,1];
p7:=P2![-2*r,2*r,1];
\end{verbatim}
We compute the linear system of conics through the cuspidal points $p_2,\,p_3$ and take
the corresponding map to $\mathbb P^3.$

\begin{verbatim}
L:=LinearSystem(LinearSystem(P2,2),[p6,p7]);
P3<a,b,c,d>:=ProjectiveSpace(K,3);
rho:=map<P2->P3|Sections(L)>;
\end{verbatim}
The image of $\mathbb P^2$ is a quadric surface $Q_2$ ($\cong\mathbb P^1\times\mathbb P^1$).

\begin{verbatim}
Q2:=rho(P2);Q2;
C:=rho(Q);C;
\end{verbatim}
There is an involution preserving both $Q_2$ and the curve $C:=\rho(Q).$

\begin{verbatim}
sigma:=map<P3->P3|[d,b,c,a]>;
C:=rho(Q);C;
sigma(Q2) eq Q2;
sigma(C) eq C;
\end{verbatim}
We compute the corresponding map to the quotient. The image of $C$ is a cubic curve, and the image of the diagonal is a conic.

\begin{verbatim}
psi:=map<P3->P2|[a+d,b,c]>;
Cu:=psi(C);
Co:=psi(Scheme(rho(P2),[a-d]));
Co:=Curve(P2,DefiningEquations(Co));
\end{verbatim}
The curve $C_u$ has a cusp singularity:

\begin{verbatim}
pts:=SingularPoints(Cu);
ResolutionGraph(Cu,pts[1]);
\end{verbatim}
The intersections of $C_o$ and $C_u:$

\begin{verbatim}
Degree(ReducedSubscheme(Co meet Cu)) eq 3;
pt:=Points(Co meet Cu)[1];
IntersectionNumber(Co,Cu,pt) eq 4;
\end{verbatim}

Let $C_1 ', C_2' $ be the fibers that intersect $M_{\PP^1 \times \PP^1}$ each at a unique point 
with multiplicity $3$.These fibers are exchanged by the involution $\sigma$ 
and are sent to a line $F_l$ which cuts the cubic curve $C_u$ at a unique point: this is a flex line. 
That line $F_l$ also cuts the conic $C_o$ at a unique point.

Conversely, let us start from the data of a conic $C_o$  and a cuspidal cubic $C_u$ 
intersecting as above, with the flex line (at the smooth flex point) of the cubic tangent to the conic. 
One can take the double cover of the plane branched over $C_o$, which is $\PP^1 \times \PP^1$. 
The pull-back of $C_u$ is then a curve satisfying the properties of Theorem \ref{thm:unicity of quartic}, 
thus the configuration $(C_o,C_u)$ we described is unique in $\PP^2$, up to projective automorphisms.

\subsection{An orbifold  ball-quotient structure from $(\PP^2,(C_o,C_u))$}

Let $C_{u}\hookrightarrow\PP^{2}$ be the unique plane cuspidal curve and let $c_1$ be its cuspidal point. 
Let $F_{l}$ be the flex line through the unique smooth flex point $c_2$ of $C_u$. 
By the previous subsection, one has the following result:
\begin{prop}
There exists a unique conic $C_{o}\hookrightarrow\PP^{2}$ such that
the following holds:\\
i) $F_{l}$  is tangent to $C_{o}$;\\
ii) $C_{o}$ cuts $C_{u}$ at points $c_{3},c_{4},c_{5}$ ($\neq c_{1},c_{2}$)
 with intersection multiplicities $4,1,1$, respectively.
\end{prop}

In this subsection we prove that there is a natural birational transformation $W\dashrightarrow \PP^2$
 such that together with the strict transform of the curves $C_o$ and $C_u$ one gets 
an orbifold ball quotient surface. 
For definitions and results on orbifold theory, we use \cite{CC,DM2} and \cite{Ulu}.

Let us blow-up over points $c_{1},c_{2},c_{3}$ and then contract
some divisors as follows (for a pictural description see figure \ref{fig:ObiW}):

We blow up over $c_{1}$ three times, the first blow-up resolves the
cusp of $C_{u}$ and the exceptional divisor intersects the strict
transform of $C_{u}$ tangentially, the second blow-up is at that
point of tangency and the third blow-up separates the strict transforms
of the first exceptional divisor and the curve $C_{u}$. One obtains
in that way a chain of $(-3)$, $(-1)$ and $(-2)$-curves. We then
contract the $(-2)$ and $(-3)$-curves obtaining in that way singularities
$A_{1}$ and $\frac{1}{3}(1,1)$. The image of the $(-1)$-curve by
that contraction map is denoted by $H$. As an orbifold we put multiplicity
$2$ on $H$. 

 We blow up over $c_{2}$ (the flex point) three times in order that the
strict transform of the curves $F_{l}$ and $C_{u}$ get separated
over $c_{2}$. We obtain in that way a chain of $(-1$), $(-2)$,
$(-2)$-curves. We then contract the two $(-2)$-curves and obtain
an $A_{2}$-singularity. The strict transform of the line $F_{l}$
is a $(-2)$-curve, which we also contract, obtaining in that way
an $A_{1}$-singularity. The contracted curve being tangent to $\tilde C_0$, 
the image $\bar C_0$ has a cusp $\mathfrak{a}_2$ at the singularity $A_1$.

We moreover blow up over $c_{3}$ four times, in order that the strict transform
of the curves $C_{o}$ and $C_{u}$ get separated over $c_{3}$. We
obtain in that way a chain of $(-1$), $(-2)$, $(-2)$, $(-2)$-curves.
We then contract the three $(-2)$-curves and obtain an $A_{3}$-singularity.
The image of the $(-1)$-curve by the contraction map is a curve denoted
by $F_{d}$, we give the weight $2$ to that curve.

\begin{figure}[h]
\caption{\label{fig:ObiW}The plane, the surfaces $Z$ and $W$}
\includegraphics[scale=0.09]{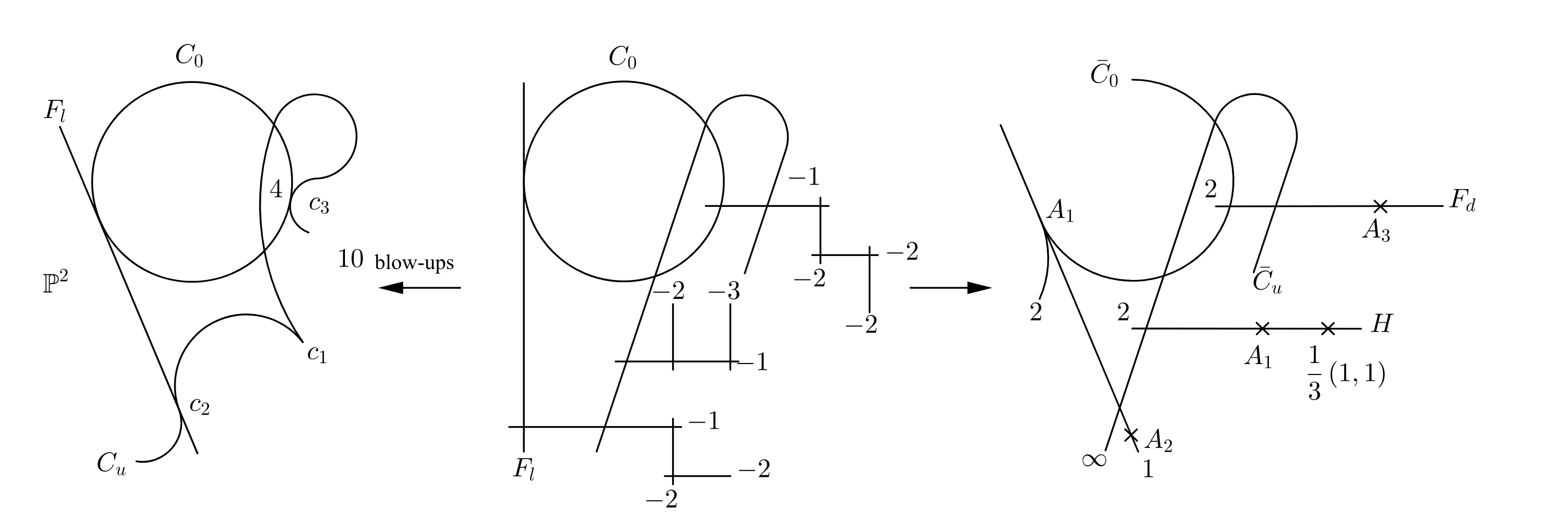}
\end{figure}

Let us denote by $W$ the resulting surface. For a curve $D$ on $\PP^2$, we denote 
by $\tilde D$ its strict transform  on $W$. Let $\mathcal{W}$ be
the orbifold with same subjacent topological space, with divisorial
part: 
\[
\Delta=\Bigl(1-\frac{1}{\infty}\Bigr)\bar C_{u}+\Bigl(1-\frac{1}{2}\Bigr)\left(\bar C_{o}+F_{d}+H\right).
\]
The singular points of $W$ are 
\[
A_{1}+A_{1}+A_{2}+A_{3}+\frac{1}{3}(1,1),
\]
and they have an isotropy $\beta$ of order $16,4,3,8,6$  respectively, for $\mathcal{W}$. 
The computation of the isotropy is immediate, except for the first point (that we shall denote by $r_1$), which is also a cusp 
on the curve $\bar C_0$ (which has weight $2$). 
Let  $SD_{16}$ be the  the  semidihedral group of order $16$, generated by the matrices $$ 
g_1=\begin{pmatrix} 
0 & -\zeta \\
-\zeta^3 & 0 
\end{pmatrix},\,  
g_2=\begin{pmatrix} 
0 & 1 \\
1 & 0 
\end{pmatrix},$$
 where $ \zeta $ is a primitive $8$th root of unity. The order $2$ elements $g_2,\,g_1^{-1}g_2g_1$ generate an order $8$ reflection group $D_4$.
 The quotient of $\CC^2$ by $SD_{16}$ has a $A_1$ singularity 
 and one computes that the image of the $4$ mirrors of $D_4$ is a curve
 with a cusp $\mathfrak{a}_2$ at the $A_1$ singularity of $\CC^2/SD_{16}$.
 The isotropy group of the point $r_1$ in the orbifold is therefore the semidihedral group $SD_{16}$ of order $16$. 
The following proposition is an application of the main result of \cite{KNS}:
\begin{prop}
The Chern numbers of the orbifold $\mathcal{W}=(W,\Delta)$ satisfy
\[
c_{1}^{2}(\mathcal{W})=3c_{2}(\mathcal{W})=\frac{9}{16},
\]
in particular $\mathcal{W}$ is an orbifold ball quotient. 
\end{prop}

\begin{proof}
Let us compute the orbifold second Chern number of $\mathcal{W}$.
We have (see e.g. \cite{RR}):
\[
\begin{array}{c}
c_{2}(\mathcal{W})=e(W)-\left((1-\frac{1}{\infty})e(\bar C_{u}\setminus S)+(1-\frac{1}{2})e(\bar C_{o}\setminus S)+\right.\\
\left.+(1-\frac{1}{2})e(F_{d}\setminus S)+(1-\frac{1}{2})e(H\setminus S)\right)-\sum_{p\in S}\left(1-\frac{1}{\beta(p)}\right),
\end{array}
\]
where $S$ is the union of the singular points of $W$ with the singular points of the round-up divisor
$\left\lceil \Delta\right\rceil $, and where moreover $\beta(p)$
is the isotropy order of the point $p$, so that for example for $p$
on $\bar C_{u}$, $\beta(p)=\infty$ and the unique point $p$ in $F_{d}$
and $\bar C_{o}$ has $\beta(p)=4$. Since we have blown-up $\PP^{2}$ over
$10$ points and we have contracted $8$ rational curves, we get
\[
e(W)=3+10-8=5.
\]
We obtain
\[
\begin{array}{c}
c_{2}(\mathcal{W})=5-\left((2-4)+\dfrac{1}{2}(2-4)+\dfrac{1}{2}(2-3)+\dfrac{1}{2}(2-3)\right)\\
-\left(10-\dfrac{1}{16}-\dfrac{1}{4}-\dfrac{1}{3}-\dfrac{1}{8}-\dfrac{1}{6}-\dfrac{1}{4}-4\cdot\dfrac{1}{\infty}\right),
\end{array}
\]
thus $c_{2}(\mathcal{W})=\frac{3}{16}$. 

Let us compute $c_{1}^{2}(\mathcal{W})$.
One has 
\[
c_{1}^{2}(\mathcal{W})=(K_{W}+\Delta)^{2},
\]
so that 
\[
\begin{array}{c}
c_{1}^{2}(\mathcal{W})=K_{W}^{2}+2K_{W}\bar C_{u}+K_{W}(\bar C_{o}+F_{d}+H)
+\frac{1}{4}(\bar C_{o}^{2}+F_{d}^{2}+H^{2})+\bar C_{u}^{2}\\[\medskipamount]
+\bar C_{u}(\bar C_{o}+F_{d}+H)+\frac{1}{2}(\bar C_{o}F_{d}+\bar C_{o}H+F_{d}H).
\end{array}
\]
Let $p:Z\to W$ be the surface above $W$ which resolves $W$ and
is a blow-up of $\PP^{2}$. Since $Z$ is obtained by $10$ blow-ups
of $\PP^{2}$ one has $K_{Z}^{2}=9-10=-1$. Moreover, since all singularities
but one are $ADE$, one has $K_{Z}=p^{*}K_{W}-\frac{1}{3}D_{1}$ where
$D_{1}$ is the $(-3)$-curve on $Z$ which is contracted to the $\frac{1}{3}(1,1)$
singularity on $W$. Since $p^{*}K_{W}\cdot D_{1}=0,$ we obtain
\[
K_{W}^{2}=-\frac{2}{3}.
\]
The curve $\bar C_{u}$ is a smooth curve of genus $0$ on the smooth locus
of $W$. The blow-up at the $\mathfrak{a}_{2}$-singularity of the
cuspidal cubic decreases the self-intersection by $4$, the remaining
blow-ups decrease the self-intersection by $1$. Since one has $4+2+3=9$
such blow-ups, one gets 
\[
\bar C_{u}^{2}=3^{2}-4-9=-4,
\]
and therefore $K_{W}\bar C_{u}=2$. Let $\tilde{D}$ be the strict transform
on $Z$ of a curve $D$ on $W$ or $\PP^2$. We have
\[
\tilde{C_{o}}=p^{*} \bar C_{o}-aF_{l}.
\]
Since $\tilde{C_{o}}F_{l}=2,$ then $a$ is equal to $1$. Since moreover $\tilde{C_{o}}^{2}=0,$
we get $0=(\tilde{C_{o}})^{2}=\bar C_{o}^{2}-2,$ thus $\bar C_{o}^{2}=2$.
We have
\[
K_{W}\bar C_{o}=(\tilde{C_{o}}+F_{l})\left(K_{W}+\frac{1}{3}D_{1}\right)=-2.
\]
Let $F_{1},F_{2},F_{3}\subset Z$ be the chain of three $(-2)$-curves above
the $A_{3}$ singularity in $W$, so that $\tilde{F_{d}}F_{1}=1$. One
computes that 
\[
\tilde{F_{d}}=p^{*}F_{d}-\frac{1}{4}\left(3F_{1}+2F_{2}+F_{3}\right)
\]
(it is easy to check that $\tilde{F}_{d}F_{1}=1,\,\tilde{F}_{d}F_{2}=\tilde{F}_{d}F_{3}=0$).
Then 
\[
-1=\tilde{F}_{d}^{2}=F_{d}^{2}-\frac{3}{4}
\]
 gives $F_{d}^{2}=-\frac{1}{4}$. One has
\[
K_{W}F_{d}=\left(K_{Z}+\dfrac{1}{3}D_{1}\right)\left(\tilde{F_{d}}+\frac{1}{4}(3F_{1}+2F_{2}+F_{3})\right)=-1.
\]
Let $D_{1},D_{2}$ be respectively the $(-3)$ and $(-2)$ curves
intersecting $\tilde{H}$. Since $\tilde{H}D_{1}=\tilde{H}D_{2}=1$, one
has
\[
\tilde{H}=p^{*}H-\frac{1}{3}D_{1}-\frac{1}{2}D_{2},
\]
thus 
\[
-1=\tilde{H}^{2}=H^{2}-\frac{1}{3}-\frac{1}{2}
\]
and $H^{2}=-\frac{1}{6}$. Moreover
\[
K_{W}H=\left(K_{Z}+\dfrac{1}{3}D_{1}\right)\left(\tilde{H}+\dfrac{1}{3}D_{1}+\dfrac{1}{2}D_{2}\right)=-1+\dfrac{1}{3}+\dfrac{1}{3}-\dfrac{1}{3}=-\dfrac{2}{3}.
\]
We compute therefore 
$$
\begin{array}{c}
c_{1}^{2}\left(\mathcal{W}\right)=-\dfrac{2}{3}+2\cdot2+\left(-2-1-\dfrac{2}{3}\right)+\dfrac{1}{4}\left(2-\dfrac{1}{4}-\dfrac{1}{6}\right)-4\\
+\left(2+1+1\right)+\dfrac{1}{2}\left(1+0+0\right)=\dfrac{9}{16},
\end{array}
$$
thus $c_{1}^{2}(\mathcal{W})=3c_{2}(\mathcal{W})=\dfrac{9}{16}$.
\end{proof}

\begin{rem}
In \cite{Deraux}, Deraux obtains $4$ different orbifold ball-quotient structures 
on surfaces birational to $A/G_{48}$. Among these, only the fourth one, $W'$, is invariant 
by the involution $\sigma$, the obstruction being the divisor $E$ in  \cite{Deraux} which creates an asymetry, unless it has weight $1$. The orbifold $\mathcal{W}$ we just described can be seen as the quotient 
of $W'$ by the involution $\sigma$. 
\end{rem}

\vspace{0.2cm} 

\noindent Vincent Koziarz 
\\Univ. Bordeaux, IMB, CNRS, UMR 5251, F-33400 Talence, France
\\{\tt vincent.koziarz@math.u-bordeaux.fr}\\

\noindent Carlos Rito
\vspace{0.1cm}
\\{\it Permanent address:}
\\ Universidade de Tr\'as-os-Montes e Alto Douro, UTAD
\\ Quinta de Prados
\\ 5000-801 Vila Real, Portugal
\\ www.utad.pt, {\tt crito@utad.pt}
\vspace{0.1cm}
\\{\it Temporary address:}
\\ Departamento de Matem\' atica
\\ Faculdade de Ci\^encias da Universidade do Porto
\\ Rua do Campo Alegre 687
\\ 4169-007 Porto, Portugal
\\ www.fc.up.pt, {\tt crito@fc.up.pt}\\

\vspace{0.1cm} 
\noindent Xavier Roulleau 
\\Aix-Marseille Universit\'e, CNRS, Centrale Marseille, I2M UMR 7373,  
\\13453 Marseille, France
\\ {\tt Xavier.Roulleau@univ-amu.fr}

\end{document}